\documentclass[10pt,twocolumn,twoside]{IEEEtran} 
\IEEEoverridecommandlockouts
\usepackage{epsf}
\usepackage{epsfig}
\usepackage{times}
\usepackage{amssymb}
\usepackage{amsmath}
\usepackage{bm}
\usepackage{amsfonts}
\usepackage{latexsym}

\usepackage{mathrsfs}
\usepackage{enumerate}
\usepackage{algpseudocode}
\usepackage{algorithm}
\usepackage{caption}

\usepackage{multirow}

\usepackage{graphicx}
\usepackage{subcaption}
\usepackage[export]{adjustbox}% http://ctan.org/pkg/adjustbox

\usepackage{trackchanges} %\showlabelstrue
\usepackage{url}
\usepackage{subcaption}
\usepackage[noadjust]{cite}

\DeclareCaptionLabelSeparator{periodspace}{.\quad}
\newcommand{\hN}{\mathcal{N}}
\newcommand{\hF}{\mathcal{F}}
\newcommand{\hP}{\mathcal{P}}

\newcommand{\hT}{\mathcal{T}}

\newtheorem{theorem}{Theorem}

\newtheorem{corollary}{Corollary}

\newtheorem{remark}{Remark}
\newtheorem{assumption}{Assumption}

\begin{document}

\title{Demand Shaping in Cellular Networks}
\author{Xinyang Zhou\ \ \ \ \ \ \ \ \ \ \ \ \ \ Lijun Chen
\thanks{X. Zhou and L. Chen are with College of Engineering and Applied Science, University of Colorado, Boulder, CO 80309, USA (emails: \{xinyang.zhou, lijun.chen\}@colorado.edu).}
\thanks{Preliminary result of this paper has been presented at the Allerton Conference on Communication, Control, and Computing, Monticello, Illinois, 2014 \cite{zhou2014demand}.}
}
\maketitle

\begin{abstract}
Demand shaping is a promising way to mitigate the wireless cellular capacity shortfall in the presence of ever-increasing wireless data demand. In this paper, we formulate demand shaping as an optimization problem that minimizes the variation in aggregate traffic. We design a distributed and randomized offline demand shaping algorithm under complete traffic information and prove its almost surely convergence. We further consider a more realistic setting where the traffic information is incomplete but the future traffic can be predicted to a certain degree of accuracy. We design an online demand shaping algorithm that updates the schedules of deferrable applications (DAs) each time when new information is available, based on solving at each timeslot an optimization problem over a shrinking horizon from the current time to the end of the day. We compare the performance of the online algorithm against the optimal offline algorithm, and provide numerical examples to complement the theoretical analysis. 
\end{abstract}

\begin{keywords}
Demand shaping, offline algorithm, online algorithm, steepest descent algorithm, supermartingale, deferrable applications, cellular networks.
\end{keywords}

\section{Introduction}\label{sec1}
We have witnessed in recent years rapid increase in demand for wireless data, driven by the proliferation of smart mobile devices. The global mobile traffic in 2016 has nearly reached 84 exabytes, more than 80 times greater than the entire global Internet traffic in 2000; yet, this number is expected to be increasing at a compound annual growth rate (CAGR) of $47\%$ in the coming five years, i.e., a seven-fold growth from 2016 to 2021 \cite{index2017global}.  However, despite frequent upgrades of cellular networks technology from 2G to 4G LTE and beyond, wireless service providers fall short of keeping up with this increasing wireless data demand, leading to congestion in the network, especially in areas of dense population. As a result, users' data rates have to be throttled to ease congestions \cite{tmobile2016,verizon2014,att2014}, at the cost of the degraded quality of service (QoS).

Admittedly, the capacity shortfall of cellular networks can be mitigated by allocating more wireless spectrum and deploying more wireless infrastructures including more and smaller cells and WiFi networks offloading, etc.  However, spectrum allocation and infrastructure upgrading are not only costly but also time-consuming, while WiFi networks may not always be available and secure. A promising alternative, inspired by the similar problem of demand response in power networks, is to improve spectrum and infrastructure efficiency through managing wireless data traffic (i.e., demand).  Notice that wireless traffic or demand usually fluctuates with a large peak-to-valley ratio throughout a day; see Fig.~\ref{fig1} for a trace of smartphone web browsing activity over a day. However, wireless capacity needs to be provisioned to meet the peak demand rather than the average. This means that the cellular network is usually stressed in peak hours while largely underutilized at other times.  If the demand profile can be shaped to reduce the peak and smooth the time variation, not only can more traffic be accommodated under limited existing capacity constraints, but also additional spectrum allocation and infrastructure upgrades can be slowed down, which together greatly improve wireless network efficiency and QoS, and yield huge savings for service providers.
\begin{figure}[t]
	\centering
  \includegraphics[trim = 0mm 20mm 0mm 40mm, clip, scale=0.5]{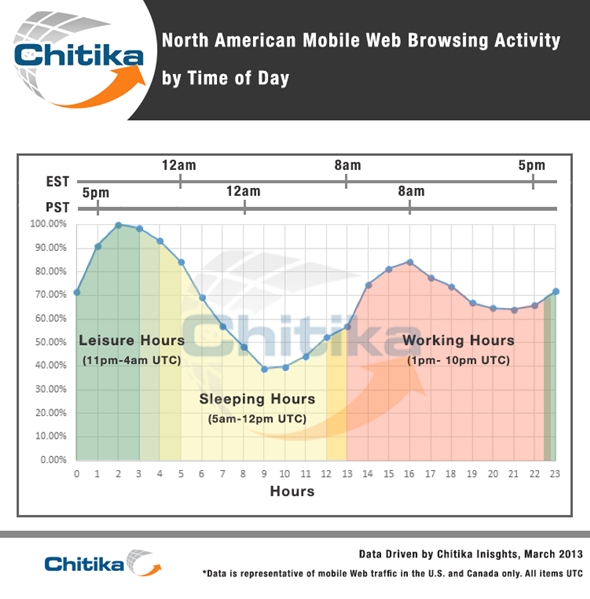}
   \caption{North America smartphone web browsing activity in one day \cite{TRAFF}.} \label{fig1}
\end{figure}

In this paper, we focus on designing demand shaping algorithms for cellular networks. We divide wireless traffic into two categories: non-deferrable traffic and deferrable traffic. Non-deferrable traffic refers to the traffic of those applications  such as  online gaming that have no or low delay tolerance, and constitutes the base traffic whose profile cannot be shaped. Deferrable traffic refers to the traffic of those applications such as file uploading/downloading that are flexible in time and only require being served by a designated deadline, e.g., finishing photo backup on cellphone by 12 am. Deferrable applications {(DAs)} are further divided into two major types: (1) continuous-rate interruptible applications such as photos backup and applications update that allow any data rates---e.g., the {\em delayed offloading} in \cite{lee2010mobile, mehmeti2014worth}, and (2) discrete-rate non-interruptible applications such as online movie streaming and video conference that usually require certain constant data rate \cite{netflixrate, skyperate} and should not be interrupted once started, e.g., one can schedule movie watching or video conference to the ``valley" time to enjoy better graphic quality and incur less data cost if he/she has the time flexibility. See Table~\ref{tab:def} for a summary of traffic types and examples. We seek to schedule the deferrable traffic to flatten the aggregate traffic profile over a day.

Specifically, we formulate the cellular traffic demand shaping as an optimization problem that minimizes the (time) variation in the aggregate traffic profile subject to the time and rate specification on each DA. We first assume complete traffic information and design an offline demand shaping algorithm. There are two challenging issues in the offline algorithm design. First, the optimization problem is non-convex because of discrete-rate non-interruptible applications. We instead solve its convex relaxation and design a randomized scheme based on the solution to the relaxed problem. Second, demand shaping involves potentially a huge number of applications and users. A centralized algorithm is not scalable. We instead design an iterative and distributed algorithm based on the descent method. We establish the almost surely convergence for the algorithm based on supermartingale theory.   

We then consider a more realistic setting with incomplete information where we can only predict future traffic to a certain degree of accuracy, and design an online and distributed demand shaping algorithm that updates the schedules of DAs each timeslot when new information  and updated prediction are available, based on the offline algorithm for an optimization problem over a shrinking horizon from the current time to the end of the day. We compare the performance of the online algorithm against the optimal offline algorithm, and provide numerical examples to complement the theoretical analysis. 

The rest of the paper is organized as follows. Section~\ref{sec:rw} briefly reviews some related work and discusses some related issues. Section~\ref{sec2} describes the system model and problem formulation. Section~\ref{sec3} presents an offline distributed algorithm for demand shaping under the assumption of complete traffic information { and characterizes its performance}. Section~\ref{sec4} considers a realistic setting of incomplete traffic information, and presents an online algorithm for demand shaping. Section~\ref{sec5} provides numerical examples to complement theoretical analysis, and Section~\ref{sec6} concludes the paper.

\begin{table}
	\begin{center}
		{
			\begin{tabular}{|| c | c ||}   
				\hline
				Traffic/Application Type & Examples\\
				\hline
				{Non-deferrable application} & Online gaming, web browsing\\
				\hline
				{Discrete-rate non-interruptible DA} & Movie streaming, video conference \\
				\hline 
				{Continuous-rate interruptible DA} & Applications update, photos backup\\
				\hline
			\end{tabular}
		}
		\caption{{Traffic/Application types and examples.}}\label{tab:def}
	\end{center}
\end{table}

\section{Related Work and Issues}\label{sec:rw}

Demand shaping in cellular networks is similar to demand response in power networks, in terms of design objectives, problem formulation, and the associated algorithmic challenges. Indeed, we borrow insights from demand response in power networks; see, e.g., \cite{LiCL11,Chen-2012-DRR, gan2012stochastic, gan2013real}.  In particular, our online demand shaping algorithm is motivated by the solution approach for online control of continuous load in reference  \cite{gan2013real}, and mathematically can be seen as its extension to incorporate discrete decision variables considered in reference \cite{gan2012stochastic}. However, our model captures realistic cellular traffic settings, as it includes both continuous and discrete decision variables. Moreover, the integration of discrete decision variables into the online algorithm makes the performance analysis of the algorithm more challenging, compared to that in \cite{gan2013real}. 
Related work also includes Zhao \emph{et al} \cite{zhao2015peak} that designs a centralized online EV charging algorithm to minimize the peak procurement from the grid under uncertain prediction of future demand and renewable energy supply, and Parise \emph{et al} \cite{parise2014mean} that  proposes a decentralized charging control for EVs to flatten the aggregate power demand profile. They all consider only continuous decision variables.

To ease the stress from high demand in cellular networks, various demand-shaping-based methodologies  as well as traffic offloading strategies have been studied in existing literatures. Tadrous \emph{et al} in \cite{tadrous2013proactive} propose a paradigm to proactively serve peak-hour
requests during the off-peak time based on prediction to smoothen the traffic demand over time without changing customers' activity pattern. However, such strategy is limited to routine behaviors only. In \cite{hajiesmaili2017incentivizing} Hajiesmaili \emph{et al} introduce an online procurement auction framework to incentivize mobile devices to participate in device-to-device load balancing to offload traffic from one heavy-loaded base station to adjacent idle ones. 
Besides, WiFi and femtocell offloading of cellular data is another major approach to easing the congestion of cellular networks; see \cite{balasubramanian2010augmenting, lee2014economics, lee2010mobile, mehmeti2014worth, iosifidis2015double, cheung2017congestion} for related works.

In this paper we have focused on designing demand shaping algorithms based on a general and simplified system model. We do not investigate the important practical issues such as the timescale and granularity at which we schedule and reschedule the DAs. We plan to develop a platform to enable automatic demand shaping in the future, and will investigate various practical issues then. Also, demand shaping involves not only the design of control algorithms but also the design of right mechanisms to incentivize the users to move out of their ``comfortable zone'' in wireless applications and data usage. Incentive design for demand shaping is currently an active research area; see, e.g., the smart data pricing in wireless networks \cite{ha2012tube, sen2013smart, zhang2016smart}, pricing design in general network service to remove congestions \cite{paschalidis2000congestion, jiang2008time}, pricing/reward signals in power distribution system \cite{zhou2017online, li2015market}, and the references therein. 

Some discussion on the practicality of demand shaping is also in place. People tend to use mobile data services whenever they want, regardless of whether it is at peak time or valley time for the cellular network. However, a survey \cite{tube2012survey} conducted in India and USA in 2012 shows that, given proper monetary incentive, many people are willing to postpone their mobile data usage, with acceptable postponement varying from minutes to hours, depending on different types of services and different individual preferences \cite{ha2012tube}. For example, wireless service providers can motivate the users to shift their demand by implementing the time-dependent pricing (TDP) strategy. TDP is now applied as a simple two-period plan by many wireless service providers around the world, in voice services and data services; e.g., Verizon \cite{verizonprice1} and Sprint \cite{sprintprice} in the US have ``happy hours'' in the night and weekend for voice service, TelCom \cite{telkomprice} in South Africa has ``Night Surfer" plans giving free data from 11pm to 5am, and Airtel \cite{airtelprice} in India provides unlimited data in the night. More refined TDP strategies can be applied to maximize benefits for both wireless service providers and users, by dynamically adjusting prices according to the data usage of the current time and predicted future. For instance, Ha {\em et al} \cite{ha2012tube} have worked on a TDP-based application named TUBE. Trials in cooperation with a local wireless service provider shows its effectiveness in shaping the traffic profile \cite{joe2011time}. Also refer to \cite{sen2012pricing} for a review of pricing strategies.

\section{System Model and Problem Formulation}\label{sec2}

Consider a cellular network that serves users for different applications such as web browsing, file sharing, real-time entertainment, etc.  The applications can be broadly divided into two categories: deferrable applications (DAs) and non-deferrable applications (non-DAs). DAs refer to those applications that are flexible in the starting time and/or data rate, while the non-DAs refer to those that should be served immediately and often have stringent data rate requirement. {Please refer to the third paragraph of Section~\ref{sec1} and TABLE~\ref{tab:def} for more detailed description and examples of DAs and non-DAs.}

This work aims to schedule the traffic of DAs so as to flatten the aggregate traffic profile over a day, subject to the time constraints and rate constraints of each application. We use a discrete-time model where one day is divided equally into $T$ timeslots, indexed by $t\in \mathcal{T}=\{1, 2, \cdots, T\}$. The duration of a timeslot can be, e.g., 30 minutes or 1 hour \cite{ha2012tube},  depending on the time resolution of scheduling decisions.  

 \begin{table}[t]
	\begin{center}
		{
			\begin{tabular}{||c|l||}   
				\hline 
				$t$ & time index, $t \in \mathcal{T} := \{1,\ldots,T\}$\\
				$n$ & DA index, $n\in\mathcal{N}:=\{1, \cdots, N\}$\\
				$\mathcal{N}'$ & set of $N'$ continuous DAs\\
				$\mathcal{N}''$ & set of $N''\!=\!N\!-\!N'$ discrete DAs\\
				$\hat{\mathcal{N}}''_t$ & set of discrete DAs started earlier\\
				$\tilde{\mathcal{N}}_t$ & set of DAs adjustable at time $t$\\
				$b$ & base traffic profile, $b=\{b(t); t\in\mathcal{T}\}$\\
				$p_n$ & data rate profile of DA $n$, $p_n=\{p_n(t); t\in\mathcal{T}\}$\\
				$\overline{p}_{n}(t)$ & upper bounds of DA $n$ on the data rate at time $t$\\
				$r_n$ & constant bit rate for DA $n\in\hN''$\\
				$l_n$ & number of timeslots to finish transmission for DA $n\in\hN''$\\
				$q$ & virtual deferrable traffic profile\\
				$d$ & average traffic profile\\
				$\hat{d}$ &average traffic profile of online ODS\\
				$\hat{d}^*$ &average traffic profile of online relaxed ODS\\
				${d}^{*}$ &average traffic profile of offline relaxed ODS\\        
				$P_n$ & total traffic required from DA $n$, $P_n=\sum_{t \in\mathcal{T}} p_n(t)$\\
				$P_n(t)$ & remaining traffic to be served for DA $n\in\hN'_t$\\
				$x^k_n$ &  change in traffic profile of DA $n$, $x^k_{n}=p_{n}^{k+1}-p_{n}^{k}$\\
				$t_n^a$ & arrival time of DA $n$\\
				$t_n^d$  & deadline of DA $n$\\ 
				$A_n$ & number of feasible profiles of DA $n\in\hN''$\\
				$f_{n,a}$ & $a$-th feasible profile of DA $n\in\hN''$\\
				$u_{n,a}$ & probability corresponding to $f_{n,a}$ \\
				$\mathcal{F}_n$ & set of all feasible traffic profiles for discrete DAs,\\
				&$\mathcal{F}_n\!=\!\{f_{n,a};1\leq a \leq A_n\}$\\
				$V(d)$ & objective value: (time) variance of d\\
				\hline 
			\end{tabular}
		}
		\caption{{Main notation.}}
	\end{center}
\end{table}

\subsection{Non-Deferrable Applications}\label{sec2A}
Non-DAs include web browsing, online gaming, and real-time chatting with multimedia, etc. The latency tolerated by these applications usually varies from hundreds of milliseconds to seconds. Since these applications should be served immediately upon request, their traffic is inelastic and constitutes the \emph{base traffic} whose profile cannot be shaped. Denote the base traffic profile by $b=\{b(t); t\in\mathcal{T}\}$. As we can only predict the base traffic to a certain accuracy, we model it as a random vector with mean $\bar{b}=\{\bar{b}(t); t \in\mathcal{T}\}$  and random derivation $\delta b=\{\delta b(t); t\in\mathcal{T}\}$ from the mean, i.e., $b=\bar{b}+\delta b$. We assume that $\delta b(t)$ has a mean of $0$ and variance of $\delta^2(t)$, and may be temporally correlated. We further assume that we can make better prediction for the timeslots that are closer to current time, modeled by a time-dependent deviation from the mean, i.e., the base traffic at some future time $\tau\in\hT$ is predicted at current time $t$ by
\begin{eqnarray}
b_t(\tau)=\bar{b}(\tau) +\delta b_t(\tau),\label{eq:ndp}
\end{eqnarray}
{where the subscript $t$ represents the timeslot when the prediction is made, and $\delta b_t(\tau)$ has a decreasing variance $\delta_t^2(\tau)$ as $t$ approaches $\tau$. More concrete model for prediction will be introduced in Section~\ref{sec5}.} The parameters $\bar{b}$ and $\delta_t$ will be specified exogenously, and can be estimated from the historical traffic records. 

\subsection{Deferrable Applications}
Assume that there are $N$ DAs in the network, indexed by $n\in\mathcal{N}=\{1, \cdots, N\}$. Each DA $n$ is characterized by an arrival time $t^a_n$ when it is requested or after which it can be started, a deadline $t^d_n$ by which its transmission must be done, and certain requirement or constraint on data rate $p_n=\{p_n(t); t\in\mathcal{T}\}$. Let $P_n$ denote the total traffic required by DA $n$, i.e., $\sum_{t \in\mathcal{T}} p_n(t)=P_n$. We can classify DAs  into two main categories: {\em continuous-rate interruptible} DAs (or continuous DAs for simplicity) that allow any data rates between certain upper and lower bounds and can be interrupted and resumed at any time before the deadline, and {\em discrete-rate non-interruptible} DAs (or discrete DAs for simplicity) that require certain (roughly) constant data rate and cannot be interrupted once they are started. For example, system backup is usually interruptible and allows any continuous data rates, while video conference is usually preferred to be non-interruptible and runs at a constant (thus discrete) data rate once it is started.

Among the total $N$ DAs, we assume there are $N'$ continuous DAs, indexed by  $n\in\mathcal{N'}=\{1, \cdots, N'\}$.  For each continuous DA, denote by $\underline{p}_{n}(t)$ and $\overline{p}_{n}(t)$ the lower and upper bounds on its data rate at time $t\in\mathcal{T}$, i.e.,
\begin{eqnarray}
\underline{p}_{n}(t) \leq p_{n}(t) \leq \overline{p}_{n}(t),\  t\in\mathcal{T}.\label{eq:rc-ci}
\end{eqnarray}
Naturally, $0\leq\underline{p}_{n}(t)\leq\overline{p}_{n}(t)$. The lower bounds $\underline{p}_{n}(t)$ are usually zero, and the upper bounds $\overline{p}_{n}(t)$ can be set according to, e.g.,  the available bandwidth. The arrival time $t_n^a$ and the deadline $t_n^d$ can be integrated into the rate constraints \eqref{eq:rc-ci} by setting $\overline{p}_{n}(t)=0$ for $t<t_n^a$ and $t > t_n^d$, i.e., no traffic is transmitted before arrival time or after deadline.

Index the rest $N''=N\!-\!N'$ discrete DAs by $n\in\mathcal{N''}=\{N'+1, \cdots, N\}$. For a discrete DA such as a streaming application, a constant bit rate $r_n$ corresponds to a certain graphic quality, { e.g., $r_n=3~\text{Mbps}$ for a SD quality movie on Netflix \cite{netflixrate}, and $r_n=1.2~\text{Mbps}$ for a HD video call on Skype \cite{skyperate}.}  As the graphic quality usually {(preferrably)} does not change during those applications, this seemingly over-simplified assumption of a single discrete rate is reasonable. 

For each DA $n\in \mathcal{N''}$ with its total traffic $P_n$ and the rate $r_n$, it takes $l_n=P_n/r_n$ consecutive timeslots (or equivalently the other way around, i.e., we calculate $P_n=l_n*r_n$ based on $l_n$ and $r_n$). Therefore, the number of its feasible traffic profiles is $A_n=t_n^d-t_n^a-l_n+1$, wherein the $a$-th feasible profile is denoted as
\begin{eqnarray}
	f_{n,a}\!=\!\Big\{p_n\Big|p_n(t)=\big\{\begin{array}{cc}
	r_n,&\!\!\!\!\text{if}~t^a_n+a-1\leq t \leq t^a_n+a+l_n\!\!\!\\
	0, & \text{otherwise}
	\end{array}\Big\}.\nonumber
\end{eqnarray} 
We denote the set of all feasible traffic profiles of DA $n\in \hN''$ by $\mathcal{F}_n=\{f_{n,a}:~1\leq a \leq A_n\}$, i.e., $p_n\in \mathcal{F}_n$, $\forall n\in\hN''$.

\begin{remark} 
All the modeled traffic parameters can be reasonably accessed or estimated in practice. For example, information regarding total required traffic $P_n$ and video streaming rate $r_n$ is available from metadata of traffic to be transmitted, parameters like $t_n^a$ and $t_n^d$ are specified by the users in advance (and $\mathcal{F}_n$ can then be calculated accordingly), whereas data rate bounds  $\underline{p}_{n}(t)$ and $\overline{p}_{n}(t)$ can be either determined by available bandwidth or designated by the users.  See, e.g., \cite{ha2012tube} for an example system involving similar information requirement and implemented with real users and service provider.\hfill$\Box$
\end{remark}

\subsection{Problem Formulation} 
We aim to schedule the traffic of DAs, so as to flatten the aggregate traffic profile as much as possible. Denote the ``average'' traffic profile by $d=\{d(t);~t \in \mathcal{T} \}:=\frac{1}{N}(b+\sum_{n\in\mathcal{N}} p_{n})$. Traffic flattening can be achieved by minimizing the time variance of $d$, formulated as the following optimal demand shaping (\textbf{ODS}) problem:

\textbf{ODS:}\vspace{-7pt}
\begin{subequations}\label{eq:ODS}
\begin{eqnarray}
\min_{p, d} && V(d)=\frac{1}{T}\sum_{t\in\mathcal{T}}\big(d(t)-\frac{1}{T}\sum_{\tau\in\mathcal{T}}d(\tau)\big)^{2}\label{eq:obj}\\ [-5pt]
\text{s.t.} && d(t) = \frac{1}{N}\big(b(t)+\sum_{n\in\mathcal{N}}p_{n}(t)\big), ~t\in\mathcal{T},\label{eq:con-ic0}\\ [-3pt]
&&\underline{p}_{n}(t)\leq p_{n}(t)\leq \overline{p}_{n}(t),~t\in\mathcal{T},~n\in\mathcal{N'}, \label{eq:con-ic1}\\
&&\sum_{t\in\mathcal{T}}p_n(t)=P_n,~n\in\mathcal{N'},\label{eq:con-ic2}\\ [-3pt]
&&p_{n} \in \mathcal{F}_n,~n\in\mathcal{N''}. \label{eq:con-nc}
\end{eqnarray}
\end{subequations}

Notice that the constraints (\ref{eq:con-nc}) for discrete DAs are non-convex.
In next section, we will investigate an {\em offline} algorithm together with a randomized scheme for solving the ODS problem under the assumption of complete information on the base traffic and DAs. Then in Section~\ref{sec4}, we will study an {\em online} algorithm for demand shaping under a more realistic setting of incomplete information where we can only predict the future traffic to a certain degree of accuracy. The offline ODS problem and algorithm will later serve as a benchmark to characterize the performance of the online algorithm.

\section{Offline Demand Shaping Algorithm}\label{sec3}
In this section, we assume complete traffic information, i.e., the base traffic and arrival of DAs are accurately known, and study how to solve the resulting {\bf offline ODS} problem. The offline problem and algorithm will provide insights into the online algorithm design for a realistic setting of incomplete information that will be considered in Section~\ref{sec4}.

\subsection{Convex Relaxation and Randomized Scheme}\label{sect:cr}

The offline ODS problem is non-convex, as each discrete DA has to pick a traffic profile from a discrete set; see constraint \eqref{eq:con-nc}.  
Consider the convex hull of $\mathcal{F}_n$, defined as
%\vspace{-8pt}
\begin{eqnarray}
 \text{conv}(\mathcal{F}_n)&:=&\Big\{{p}_{n} |~{p}_{n}=\sum_{a=1}^{A_n}u_{n,a}\cdot f_{n,a}, ~ u_{a,n}\geq 0\nonumber\\ [-6pt]
 &&~~\text{and}~\sum_{a=1}^{A_n}u_{n,a}=1\Big\},\label{eq:rf}
\end{eqnarray}
{where $u_{n}:=\{u_{n,1},\ldots,u_{n,A_n}\}$ is the convex combination coefficients, and will be interpreted as probability distribution in the randomized algorithm to be introduced soon.}
We will instead solve the convex relaxation of the ODS problem by replacing \eqref{eq:con-nc} with the following constraint:
\begin{eqnarray}
p_{n} \in \text{conv}(\mathcal{F}_n),~n\in\mathcal{N''}.\label{eq:con-ncr}
\end{eqnarray}
We call the relaxed problem (\ref{eq:obj})--(\ref{eq:con-ic2})(\ref{eq:con-ncr}) the {\bf R-ODS} problem. However, a solution $p_n^*{ \in\text{conv}(\mathcal{F}_n)},~n\in\mathcal{N''}$ to the R-ODS problem might not be feasible for original ODS, i.e., $p_n^*\notin \mathcal{F}_n$. But since by definition \eqref{eq:rf} a solution $p_{n}^*$ can always be written as the convex combination $\sum_{a=1}^{A_n}u_{n,a} f_{n,a}$ we will randomly pick a traffic profile $p_n=f_{n,a} \in \mathcal{F}_n$ with corresponding probability $u_{n,a}$.  That said, we will design a randomized  algorithm for the offline ODS problem, based on the solution  to the R-ODS problem. We will integrate it into a distributed algorithm next.

\subsection{Distributed Algorithm}\label{sect:da}
Solving the R-ODS problem (and the offline ODS problem) directly in a centralized way requires collecting information on all DAs, which may incur too much communication overhead and is impractical in the real network. Moreover, the users may not be willing to reveal information on DAs due to privacy concern. 
Therefore, we seek to solve it in a distributed way.  Noticing that R-ODS problem has decoupled constraints, we attempt to design an {\em iterative and distributed} algorithm based on the decent method \cite{Boyd}. 

Before deriving the algorithm, we establish the following useful results. At {\em k}-th iteration, let $p^k=\{p_n^k;~n\in\mathcal{N}\}$ be the traffic profiles of all DAs, 
$d^k=\frac{1}{N}(b+\sum_{n\in\mathcal{N}}p_{n}^{k})$ the average traffic profile, and $x^k_{n}=p_{n}^{k+1}-p_{n}^{k},~n\in\mathcal{N}$ the change in traffic profile of DA $n$ from iteration $k$ to $k+1$. We have:
\begin{eqnarray}
E\big[\big\|\sum_{n\in\mathcal{N}}x^k_{n}\big\|_2^2\big]=\sum_{n\in\mathcal{N}}\!Var(x^k_{n})+\big\| \sum_{n\in\mathcal{N}}E[x^k_{n}] \big\|_2^2,
\end{eqnarray}
where the variance  $Var(x_n^k):=E\big[\| x_n^k \| ^2_2 \big]\!-\!\| E[x_n^k] \| ^2_2$, and $E[\cdot]$ denotes the average.\footnote{Notice that we consider a randomized scheme only for discrete DAs. That said, for continuousDAs there is no randomness and their variance is zero.}  
By Jensen's inequality, 
\begin{eqnarray}
\| \sum_{n\in\hN}E[x^k_{n}] \|^2_2 \leq N\sum_{n\in\hN}\big\| E[x^k_n]\big\|^2_2.
\end{eqnarray}
Therefore, one has
\begin{eqnarray}
E\big[\|\sum_{n\in\mathcal{N}}x^k_{n}\|_2^2\big] \leq  \sum_{n\in\hN}Var(x^k_{n})+N\sum_{n\in\hN}\big\| E[x^k_n]\big\|^2_2.\label{eq:var}
\end{eqnarray}
And it follows that
\begin{eqnarray}
&& TN^2\big(E[V(d^{k+1})|p^k]-V(d^k)\big) \nonumber\\
&=&E\big[\|\sum_{n\in\hN}x^k_{n}\|^2_2+2\langle N d^{k},\!\sum_{n\in\hN}\!\!x^k_{n}\rangle\big]\nonumber\\ [-3pt]
&\leq& \!\!\!\!\sum_{n\in\hN}\!Var(x^k_{n})+N\!\sum_{n\in\hN}\!\| E[x^k_n]\|^2_2+2\!\sum_{n\in\hN}\!\!E[\langle N d^{k},x^k_{n}\rangle]\nonumber
\end{eqnarray}
\begin{eqnarray}
&=& \!\!\!\!\sum_{n\in\hN'}\big(2\langle N d^{k},x^k_{n}\rangle+N\| x^k_n\|^2_2\big)+\!\!\sum_{n\in\hN''}\!\!\big(2\langle N d^{k},E[x^k_{n}]\rangle\nonumber\\ [-3pt]
&&~~+N\| E[x^k_n]\|^2_2 +Var(x^k_{n})\big).\label{eq:ld}
\end{eqnarray} 
Denote by $W_1$ the first term in \eqref{eq:ld} and $W_2$ the second. For $n\in\mathcal{N'}$, we choose $p_n^{k+1}$ so as to minimize $W_1$, i.e., to solve
\vspace{-4mm}
\begin{subequations}\label{eq:algc}
\begin{eqnarray}
\min_{p_n} &&  2\langle d^{k},p_{n}-p_n^k\rangle+\| p_{n}-p_n^k\|^2_2 \label{eq:algc-ob}\\
\text{s.t.} && \eqref{eq:con-ic1}-\eqref{eq:con-ic2}. \label{eq:algc-con}
\end{eqnarray}
\end{subequations}

On the other hand, after some mathematical manipulations, we have
\vspace{-0.5mm}
\begin{eqnarray}
\nonumber W_2\!\!\!\!\! &=&\!\!\!\!\!\!\! \sum_{n\in\hN''}\!\!\big(2N \langle d^k\!-p_n^k, E[p_n^{k+1}]\rangle\! + \!(N\!-\!1)\|E[p_n^{k+1}]\|^2_2 \big)\!\! +\!\Pi^k, \label{eq:w2}
\end{eqnarray}
where $\Pi^k$ is a constant given $p_n^k$. For $n\in\mathcal{N''}$, we choose $p_n^{* k+1}$
so as to minimize $W_2$, i.e., to solve
\vspace{-0.5mm}
\begin{eqnarray}
\min_{p_n\in\text{conv}(\mathcal{F}_n)} 2 \langle d^k-p_n^k, p_n\rangle + \frac{N-1}{N}\|p_n\|^2_2.\label{eq:algd}
\end{eqnarray}

In essence, what we have done is to \emph{maximize the expected incremental decrease in the objective value $V(d)$} at each iteration (i.e., steepest descent). This motivates a distributed demand shaping algorithm with the collaboration of a coordinator; see Algorithm \ref{alg1}. The wireless service provider can implement a logical coordinator at the base station.

\begin{algorithm}
  \caption{Offline Demand Shaping  (Off-DS) Algorithm}
 \vspace{1mm} 
{At $k$-th iteration:}
\label{alg1}
\begin{enumerate}
\item Upon gathering traffic profiles $p_n^k$ from DAs, the coordinator calculates the average traffic profile $d^{k}= \frac{1}{N}(b+\sum_{n\in\hN}p^{k}_{n})$ and announces it to DAs (or the end users) over a signaling or control channel.
\item Upon receiving the average traffic profile $d^k$, 
\begin{itemize}
\item DA $n \in \mathcal{N'}$ updates its traffic profile by
\begin{eqnarray}
\nonumber p_n^{k+1} = \underset{p_n}{\arg\min} && \hspace{-5mm}\big\| p_n-p_n^k+d^k\big\|_2^2\\ [-3pt]
\nonumber \text{s.t.} && \hspace{-5mm}  \text{\eqref{eq:con-ic1}--\eqref{eq:con-ic2}},
\end{eqnarray}
and submits it to the coordinator.
\item DA $n \in \mathcal{N''}$ calculates the average traffic profile by
\begin{eqnarray}
\hspace{-2mm}\nonumber p_n^{* k+1}=\underset{p_n\in\text{conv}(\mathcal{F}_n)}{\arg\min} && \hspace{-6mm} \Big\| p_n- \frac{N}{N-1}( p_n^{k}-d^k)\Big\|_2^2,
\end{eqnarray}
which is $p_n^{* k+1}=\sum_{a=1}^{A_n}u_{n,a}^{k+1} f_{n,a}$, and then randomly chooses a traffic profile $p_n^{k+1}=f_{n,a}$ with probability $u_{n,a}^{k+1} $ and submits it to the coordinator. 
\end{itemize}
\end{enumerate}
\end{algorithm}

The Off-DS algorithm is a distributed algorithm wherein each DA solves its own simple optimization problem based on its previous decision, the average traffic profile $d^k$, and local constraints, while the coordinator collects the proposed traffic profiles and updates the average traffic profile. Therefore, this algorithm is not only preserving privacy of the users, but also scalable and thus capable of quick response, which is crucial especially in real-time implementation in Section~\ref{sec4}.

The computational complexity of the Off-DS algorithm is estimated as follows for completeness. Given certain accuracy requirement  $\epsilon>0$ in the objective function value, the descent method requires $O(\log(1/\epsilon))$ iterations \cite{Boyd}. At each iteration, DAs solves an easy quadratic programming with a polynomial complexity of $O(T^{O(1)})$ \cite{potra2000interior}. On the other hand, the coordinator calculates the average traffic profile which requires $O(N)$ complexity each iteration. As a result, the Off-DS algorithm requires overall computational complexity of $O\big((N+T^{O(1)})\log(1/\epsilon)\big)$.

\begin{remark}\label{rem:1} For simpler expression, we use $p_n$ as the decision variable for DA $n\in\hN''$ in algorithm design and analysis, while in real implementation, it is more convenient to use probability distribution $u_n$ as the equivalent decision variable. Also notice that, if there is no continuous DA, Algorithm \ref{alg1} reduces to the stochastic algorithm in \cite{gan2012stochastic}. We expect that the solution approach---randomized algorithm based on the ``steepest'' descent method for the convex relaxed problem---that we lay out in Sections~\ref{sect:cr} and \ref{sect:da} will find broad application in designing efficient algorithms for optimization problems that involve both continuous and discrete decision variables. \hfill$\Box$
\end{remark}

\subsection{Convergence}

Before showing the convergence of the Off-DS algorithm, we first establish two useful relations. For each DA $n \in \mathcal{N}^\prime$, since $p_n^{k+1}$ solves the problem (\ref{eq:algc}), we have the first-order optimality condition
\begin{eqnarray}
\langle p_n^{k+1}-p_n^k+d^k,p_n-p_n^{k+1}\rangle \geq 0 \label{optconw1}
\end{eqnarray}
for any feasible $p_n$. Set $p_n=p_n^k$ to obtain
\begin{eqnarray}
\langle d^k,p_n^{k+1}-p_n^k\rangle \leq -\| p_n^{k+1}-p_n^k\|_2^2.\label{eq:doc}
\end{eqnarray}
For each DA $n \in \mathcal{N''}$, recalling that $p_n^{*k+1}=E[p_{n}^{k+1}]$,
by the first-oder optimality condition, we have 
\begin{eqnarray}
\langle \frac{N}{N-1}( d^{k}-p_n^k)+p_n^{*k+1},p_n-p_n^{*k+1}\rangle \geq 0 \label{optconw2}
\end{eqnarray}
for any feasible $p_n$. Set $p_n=p_n^k$ to get
\begin{eqnarray}
\langle N d^{k},p_n^{*k+1}-p_n^{k}\rangle \!&\leq&\! -(N-1)\| p_n^{*k+1}-p_n^{k}\|^2_2\nonumber\\ [-1pt]
&&~~+\langle p_n^k,p_n^{*k+1}-p_n^k\rangle. \label{eq:dod}
\end{eqnarray}

Now, construct a filtration $\Sigma^*$ of the probability space $\{\Omega, \Sigma, \hP\}$, where the sample space $\Omega$ is the feasible set specified by the constraints \eqref{eq:con-ic1}--\eqref{eq:con-nc}, the $\sigma$-algebra $\Sigma_k = \Omega,~ k\geq 0$, and $\hP(\Sigma_k)=\{\delta (p_n-p_n^k), n\in\hN';~u_{n,a}^k, 1\leq a\leq A_n, n\in\hN''\}$, i.e., determined by the $k$-th iteration of the Off-DS algorithm.

\begin{theorem}\label{thm:thm1} 
The pair $(V(d),~\Sigma^*)$ is a supermartingale. \hfill $\Box$ 
\end{theorem}
\begin{proof}
First, notice that $V(d)$ is bounded from below.  So, $E[-\min \{0, V(d)\}]<\infty$. Second, applying relations \eqref{eq:doc}--\eqref{eq:dod} to equation \eqref{eq:ld}, we obtain
\begin{eqnarray}
&& TN^2\big(E[V(d^{k+1})| p^k]-V(d^k)\big) \nonumber\\
&\leq& \sum_{n\in\hN'}-N\| x^k_n\|^2_2+\sum_{n\in\hN''}\big(Var(x^k_n)\nonumber\\ [-4pt]
&&~~+(-N+2)\big\| E[x^k_n]\big\|^2_2+2\langle p_n^k,p_n^{*k+1}-p_n^k\rangle\big)\nonumber\\
&=&\sum_{n\in\hN'} -N\| x^k_n\|^2_2+\sum_{n\in\hN''}(-N+1)\big\| E[x^k_n]\big\|^2_2\nonumber\\ [-1pt]
\nonumber &\leq&0\label{eq:pr0},
\end{eqnarray}
i.e., $ E[V(d^{k+1})| p^k]\leq V(d^k)$.  By definition, $(V(d),~\Sigma^*)$ is a supermartingale \cite{Grimmett}.
\end{proof}

Notice that $(V(d),~\Sigma^*)$ is a nonnegative supermartingale. By the martingale convergence theorem \cite{Grimmett}, the following result is immediate.
\begin{corollary}\label{thm:coro1}
$V(d^\infty)=\lim_{k\to\infty} V(d^k)$  exists almost surely, where $V(d^{\infty})$ is some random variable.\hfill $\Box$ 
\end{corollary}

\begin{theorem}\label{thm:thm2}
Denote by $\hP^\infty$ an ``equilibrium'' distribution over traffic profiles that $(V(d),~\Sigma^*)$ converges to. The support of $\hP^\infty$ is a singleton. \hfill $\Box$ 
\end{theorem}
\begin{proof}
When $(V(d),~\Sigma^*)$ converges, $ E[V(d^{k+1})| p^k] = V(d^k)$. This requires $E[x^k_n]=E[x^k_{n'}], ~n, n' \in\hN$,  $p_n^{k+1}=p_n^k,~n\in\hN'$, and $p_n^{* k+1}=p_n^k,~n\in\hN''$ for \eqref{eq:var},  \eqref{eq:doc}, and  \eqref{eq:dod} to hold with equality. Notice that $p_n^{* k+1}=p_n^k$ implies $p_n^{ k+1}=p_n^k$, as different feasible traffic profiles of DA $n\in\hN''$ are linearly independent. Thus, $p_n^{k+1}=p_n^k,~n\in\hN$. So, the support of $\hP^\infty$ contains only one point.
\end{proof}

Denote by $p^\infty$ an ``equilibrium'' traffic profile of the Off-DS algorithm, i.e., if $p^k=p^\infty$, then $p^{k+1}=p^\infty$. Obviously the set of equilibrium profiles is not empty, as an optimum of the offline ODS problem is an equilibrium. The following result follows immediately from Theorem \ref{thm:thm2} and Corollary \ref{thm:coro1}.
\begin{theorem}\label{thm:thm3}
The Off-DS algorithm converges almost surely to an equilibrium traffic profile. \hfill $\Box$ 
\end{theorem}

By equations (\ref{optconw1})--(\ref{optconw2}), we have the following optimality conditions at equilibrium $p^{\infty}$: for any feasible $p_n$, 
\begin{subequations}\label{eq:optconw34}
\begin{eqnarray}
\langle b+\sum_{m\in\hN}p^{\infty}_m, p_n-p_n^{\infty}\rangle \geq 0,\ n\in\hN',\label{optconw3}\\ [-3pt]
\langle b+\sum_{m\neq n}p^{\infty}_m, p_n-p_n^{\infty}\rangle \geq 0,\ n\in\hN''.\label{optconw4}
\end{eqnarray}
\end{subequations}

\subsection{Performance Analysis of the Offline Algorithm}
We now characterize the performance of Off-DS algorithm with respect to the relaxed problem R-ODS that at optimum may attain a lower objective value than the ODS problem. Specifically, denote by $p^*$ the solution of R-ODS, we bound the gap between the equilibrium of the Off-DS algorithm and the solution of the R-ODS problem as:  $G^{\text{off}}:=V(d^{\infty})-V(d^*)$, where $d^{\infty} = (b+\sum_{n\in\mathcal{N}}p^{\infty}_{n})/N$ and $d^{*} = (b+\sum_{n\in\mathcal{N}}p^{*}_{n})/N$. 
Denote by $G_r^{\text{off}}:=(V(d^{\infty})-V(d^*))/V(d^*)$ the relative gap achieved by the Off-DS algorithm.

\begin{theorem}\label{the:peroffline}
For the Off-DS algorithm, the gap $G^{\text{off}}$ is bounded as follows: 
\begin{eqnarray}
G^{\text{off}}\leq \frac{2}{TN^2}\sum_{n\in\hN''}\| p_n^{\infty}\|_2^2.
\end{eqnarray}
Moreover, the relative gap diminishes as the number $N''$ of discrete DAs increases, i.e.,
\begin{eqnarray}
\lim_{N''\rightarrow\infty} G_r^{\text{off}}=0.\label{eq:dimishsub}
\end{eqnarray}
 \hfill $\Box$ 
\end{theorem}
\begin{proof}
For notational simplicity, let $c_d:=\sum_{t\in\mathcal{T}}d(t)/T$, which is a constant given the total amount of traffic. The objective value can be written as 
\begin{eqnarray}
V(d)&=&\frac{1}{T}\|d-c_d\cdot \bm{1}\|^2_2=\frac{1}{T}(\|d\|_2^2+c_d^2\|\bm{1}\|_2^2-2\langle d,\bm{1} \rangle)\nonumber\\
&=&\frac{1}{T}(\|d\|_2^2+T\cdot c_d^2-2T\cdot c_d),\nonumber
\end{eqnarray}
where only the part $\|d\|_2^2$ contains decision variables. We can thus write the gap $G^{\text{off}}$ as
\begin{eqnarray}
&&\hspace{-.3cm}G^{\text{off}}=V(d^{\infty})-V(d^*)= \frac{1}{T}\big(\|d^{\infty}\|_2^2-\|d^{*}\|_2^2\big)\nonumber\\ [-6pt]
&=& \frac{1}{T}\big(-\|d^{\infty}-d^*\|^2_2+\langle 2d^{\infty}, d^{\infty}-d^* \rangle\big)\nonumber\\[-1pt]
&\leq& \frac{1}{T}\langle 2d^{\infty}, d^{\infty}-d^* \rangle\nonumber\\ [-3pt]
&=&\frac{2}{N^2}\Big(\!\!\sum_{n\in\hN'}\!\langle Nd^{\infty}, p_n^{\infty}-p_n^*\rangle+\!\!\sum_{n\in\hN''}\!\langle Nd^{\infty}, p_n^{\infty}-p_n^*\rangle\Big)\nonumber\\ [-4pt]
&\leq& \frac{2}{TN^2}\sum_{n\in\hN''}\langle p_n^{\infty}, p_n^{\infty}-p_n^*\rangle\nonumber\\ [-4pt]
&\leq& \frac{2}{TN^2}\sum_{n\in\hN''}\| p_n^{\infty}\|_2^2,\nonumber
\end{eqnarray}
where the second inequality follows from (\ref{eq:optconw34}). Note that $\| p_n^{\infty}\|_2^2$ is a constant for $n\in\hN''$.
Then the relative gap $G_r^{\text{off}}$ can be bounded as
\begin{eqnarray}
G_r^{\text{off}}&\leq&  \frac{2}{TN^2}\sum_{n\in\hN''}\| p_n^{\infty}\|_2^2/V(d^*)\nonumber\\ [-4pt]
&=&\frac{\sum_{n\in\hN''}\| p_n^{\infty}\|_2^2}{\|b+\sum_{n\in\hN}p^*_n\|^2_2+N^2(T\cdot c_d^2-2T\cdot c_d)},
\label{eq:peroff}
\end{eqnarray}
whose numerator increases linearly with $N''$ and denominator increases linearly with the square of $N''$. Equation (\ref{eq:dimishsub}) follows.
\end{proof}

\begin{remark}
We use the relaxed problem R-ODS for comparison instead of the ODS problem for two reasons. First, it is difficult to characterize the optimum of the non-convex ODS problem, and thus evaluating the gap between the equilibrium of the Off-DS algorithm and the optimum of ODS problem is mathematically hard. Second, R-ODS achieves an optimal objective value that is not greater than ODS, resulted from convex relaxation for the discrete decision variables. Therefore, $G^{\text{off}}$ provides an upper bound for the ``actual'' sub-optimality, i.e., the gap between the equilibrium of Off-DS and the optimum of ODS.\hfill$\Box$
\end{remark}

\section{Online Demand Shaping Algorithm}\label{sec4}

In this section, we consider a realistic setting with incomplete information where we can only predict future traffic to a certain degree of accuracy, and study online demand shaping that makes decisions based on the prediction of future traffic and updates the decision as new information becomes available.

A typical algorithm used in this setting is the receding horizon control; see, e.g., \cite{RHC}. However, as the objective function \eqref{eq:obj} does not have a nice additive structure, receding horizon control algorithm does not admit an easy analysis. We will instead extend a shrinking horizon control algorithm, which is used in \cite{gan2013real} that studies mathematically the same problem with only continuous DAs, to include discrete DAs, and apply it to our online demand shaping ({\bf online DS}) problem.

\subsection{Online Algorithm}\label{sec4A}

We assume that the number $m_t$ of DAs arriving at time $t$ is randomly distributed with a mean $\lambda_t$ and variance $(\delta\lambda_t)^2$, and the total amount of traffic of each DA is randomly distributed with a mean $P$ and variance $(\delta P)^2$. 
Denote by $\mathcal{N}_t'=\{1, \cdots, N_t'\}$  the set of continuous DAs and $\mathcal{N}_t ''=\{N'+1, \cdots, N_t\}$
the set of discrete DAs that have arrived by time $t\in\mathcal{T}$, and let $\mathcal{N}_t=\mathcal{N}_t '\cup\mathcal{N}_t''$  and $N''_t=N_t-N'_t$. Notice that we cannot reschedule the remaining traffic of a discrete DA that has already started. Denote by $\tilde{\hN}_t''\subseteq \hN_t''$ the set of discrete DAs that have not been started by time $t$. For DA $n\in\tilde{\hN}_t''$, denote by $\hF_n(t)=\{f_{n,a};1\leq a \leq A_n(t)\}$ the set of feasible traffic profiles at time $t$. Let $\tilde{\mathcal{N}}_t=\mathcal{N}_t '\cup\tilde{\mathcal{N}}_t''$ be the set of DAs whose profiles are still adjustable at time $t$ (i.e., all the continuous DAs and the discrete DAs that have not started by  time $t$).

At time $t$, we make a prediction $b_t(t:T)$ of base 
traffic for the rest timeslots of the day, and we also have the information
on DA $n\in \hN_t$ and the expected total future deferrable
traffic $\sum_{\tau=t+1}^T P\lambda_{\tau}$. Following \cite{gan2013real}, we introduce a {\em virtual} deferrable traffic profile $q(t: T)=\{q(\tau); t\leq \tau \leq T\}$ with $q(t)=0$ and $\sum_{\tau=t}^T q(\tau) = \sum_{\tau=t+1}^T P\lambda_{\tau}$,  to emulate the impact of the future deferrable traffic upon the current demand shaping decision.  With the afore setup,  we aim to schedule and reschedule the DAs, so as to solve the following problem at each timeslot $t\in\hT$.
\begin{subequations}
\begin{eqnarray}
\textbf{ODS$_t$:}\hspace{-8mm}&&\nonumber\\
\min \hspace{-5mm}&& V(d)=\frac{1}{T\!-\!t\!+\!1}\sum_{\tau=t}^{T}\Big(d(\tau)-\frac{\sum_{s=t}^{T}d(s)}{T\!-\!t\!+\!1}\Big)^{2}\label{eq:objt}\\
\nonumber \text{over}   \hspace{-5mm}&&    p(t:T),~ d(t:T),~ q(t:T)\\
\text{s.t.} \hspace{-5mm}&& d(\tau)=\frac{b_t(\tau)\!+\!q(\tau)\!+\!\sum_{n\in\mathcal{N}_t}p_{n}(\tau)}{N_t}, ~\tau \geq t, \label{eq:con-at} \\
\hspace{-5mm}&&\underline{p}_{n}(\tau) \leq p_{n}(\tau)\leq \overline{p}_{n}(\tau), \tau\geq t, ~n\in\mathcal{N}_t', \label{eq:con-c1t}\\[-1pt]
\hspace{-5mm}&&\sum_{\tau=t}^{T}p_{n}(\tau)=P_{n}(t), ~n\in\mathcal{N}_t',\label{eq:con-c2t}\\[-4pt]
\hspace{-5mm}&&p_{n} \in \mathcal{F}_n(t), ~n\in\tilde{\mathcal{N}}_t'', \label{eq:con-dt} \\ [-2pt]
\hspace{-5mm}&&\sum_{\tau=t}^{T}q(\tau)=\sum_{\tau=t+1}^T P\lambda_{\tau}\label{eq:con-vl},
\end{eqnarray}
\end{subequations}
where $p(t:T)=\{p_n (\tau); t\leq \tau\leq T, n\in \tilde{\hN}_t\}$, $d(t:T)=\{d(\tau); t\leq \tau\leq T\}$, and $P_n(t)=P_n-\sum_{\tau=1}^{t-1}p_n(\tau), n\in\hN_t'$ is the amount of traffic to be served at or after time $t$.

We can solve the ODS$_t$ problem at each timeslot the same way as we solve the offline ODS problem \eqref{eq:ODS}, constituting an online demand shaping algorithm; see Algorithm \ref{alg2}, wherein the convergence { (and computational complexity)} of Step 2) can be established {(and analyzed)} in the same way as Algorithm~\ref{alg1}.

\begin{algorithm}[t]
\caption{Online Demand Shaping (On-DS) Algorithm}
\label{alg2}
At each timeslot $t\in\hT$: 
\begin{enumerate}
\item Denote by $p_n^{(t-1)}, n\in\hN_{t-1}$ the schedules determined by time $t-1$, and by $\hat{\hN}_t''\subseteq \hN_t''$ the set of discrete DAs that has been started before time $t$. For each DA $n\in\hat{\hN}_t''$, set its schedule $p_n(t; T)=\{p_n(\tau); t \leq \tau \leq T\}$ as 
$p_n(\tau)= p_n^{(t-1)}(\tau),~t \leq \tau \leq T$.
\item Solve the ODS$_t$ problem iteratively: at $k$-th iteration,
\begin{enumerate}
\item Upon gathering traffic profiles $p_n^k (t: T)=\{p_n^k (\tau); t\leq \tau \leq T\}$ from DAs $n\in\tilde{\hN}_t$, the coordinator solves \begin{eqnarray}
\nonumber \hspace{-7mm}\min_{q(t+1: T)} \hspace{-3.5mm}&&\hspace{-3.5mm} \sum_{\tau=t+1}^{T}\Big(b_t(\tau)+q(\tau)+\!\sum_{n\in\hat{\mathcal{N}}_t''}p_{n}(\tau)+\sum_{n\in\tilde{\mathcal{N}}_t}p_{n}^k(\tau)\Big)^{2}\\[-4pt]
\nonumber \hspace{-7mm}\text{s.t.}  \hspace{-0mm}&&\hspace{-2mm} \eqref{eq:con-vl},
\end{eqnarray}
to obtain a virtual deferrable traffic $\{q^k (\tau); t+1 \leq \tau \leq T\}$, and then calculates the average traffic 
$d^{k} (\tau) = \frac{1}{N_t}\big(b_t(\tau)+q^k(\tau)+\sum_{n\in\hat{\mathcal{N}}_t''}p_{n}(\tau)+\sum_{n\in\tilde{\mathcal{N}}_t}p_{n}^k(\tau)\big)$ for $\tau\geq t$ 
and announces it to DA $n\in \tilde{\hN}_t$ over a signaling or control channel. 
\item Upon receiving the average traffic profile $d^k$,  
\begin{itemize}
\item DA $n \in \mathcal{N}_t'$ obtains $p_n^{k+1}(t:T)$ by
\begin{eqnarray}
\nonumber \hspace{-0mm} \min_{p_n(t:T)} && \hspace{-5mm}\big\| p_n(t:T)-p_n^k(t:T)+d^k(t:T)\big\|_2^2\\[-4pt]
\nonumber \hspace{-0mm} \text{s.t.} && \hspace{-5mm}  \text{\eqref{eq:con-c1t}--\eqref{eq:con-c2t}},
\end{eqnarray}
and submits the updated profile to the coordinator.
\item DA $n \in \tilde{\mathcal{N}}_t''$ calculates $p_n^{* k+1}(t:T)$ by
\begin{eqnarray}
\nonumber \hspace{-8mm} \min_{p_n(t:T)} && \hspace{-3mm} \Big\| p_n(t:T)- \frac{N_t}{N_t-1}( p_n^{k}(t:T)-d^k (t:T))\Big\|_2^2\\[-4pt]
\nonumber \hspace{-8mm} \text{s.t.}&&\hspace{-3mm} p_n(t:T) \in\text{conv}(\mathcal{F}_n(t)),~n\in\tilde{\mathcal{N}}_t'',
\end{eqnarray}
represents it as a convex combination $p_n^{* k+1}=\sum_{a=1}^{A_n(t)}u_{n,a}^{k+1} f_{n,a}$, and randomly chooses a traffic profile $p_n^{k+1}=f_{n,a}$ with probability $u_{n,a}^{k+1} $ and submits it to the coordinator. 
\end{itemize}
\end{enumerate}
\end{enumerate}
\end{algorithm}
  
\subsection{Performance Analysis of the Online Algorithm}\label{sec:perf}
We now characterize the performance of On-DS algorithm with respect to the result of Off-DS algorithm % (i.e., Offline ODS problem) 
which serves as a benchmark. We will make the following assumptions to simplify the analysis and obtain insights into how uncertainties affect the performance of On-DS algorithm.

\begin{assumption}\label{ass2}
	The amount of deferrable traffic is large {and flexible} enough so that a valley-filling schedule exists at every time $t=1,\ldots,T$, i.e., there exists some constant $C(t)\geq b_t(\tau), \forall\tau=t,\ldots,T$ such that
	\begin{eqnarray}
	&&\hspace{-.3cm}Nd(t)=C(t)\nonumber\\[-4pt]
	&=&\!\!\!\frac{1}{T-t+1}\!\Big(\sum_{\tau=t}^{T}b_t(\tau)+\!\!\sum_{\tau=t+1}^{T}\!\!P\lambda_{\tau}+\sum_{n=1}^{N_t}P_n(t)\!\Big)\!.
	\end{eqnarray}
	\hfill$\Box$
\end{assumption}

\begin{remark}
Assumption~\ref{ass2} looks a strong assumption, and we do not have empirical evidence to support it as demand shaping has not being widely adopted in current cellular networks. However, with increasing penetration of deferrable traffics and users, this assumption expects to hold.  One purpose of algorithm design as in this paper and incentive design as in~\cite{ha2012tube} is to facilitate and incentivize wide adoption of demand shaping. On the other hand, valley-filling represents the scenario where demand shaping is most useful and presents a benchmark for the potential of demand shaping. Mathematically, it is very difficult to analyze the performance of the online algorithm under more general assumption than Assumption~\ref{ass2}. However, notice that in numerical examples in Section~\ref{sec5}, we do not impose Assumption~\ref{ass2} while the results still fall into the bound specified in Theorem~\ref{the:peronline}.\hfill$\Box$
\end{remark}

\begin{assumption}\label{ass3}
	The base traffic prediction at $t$ is modeled as the following causal filter
	\vspace{-3pt}
	\begin{eqnarray}
		b_t(\tau)=\bar{b}(\tau)+\sum_{s=1}^T e(s)f(\tau-s),\ \tau=1,\ldots,T,
	\end{eqnarray}
	where $e=\{e(s)\}_{s=1}^T$ is an uncorrelated sequence of independent and identically distributed random variables with mean 0 and variance $\delta^2$, and $f=\{f(\tau)\}_{\tau=-\infty}^{\infty}$ is the impulse response with $f(0)=1$. Let $F(t):=\sum_{s=0}^{t}f(s)$.\hfill$\Box$
\end{assumption}
 
\begin{figure}
	\begin{center}
		\includegraphics[trim = 0mm 0mm 0mm 0mm, clip, scale=0.27]{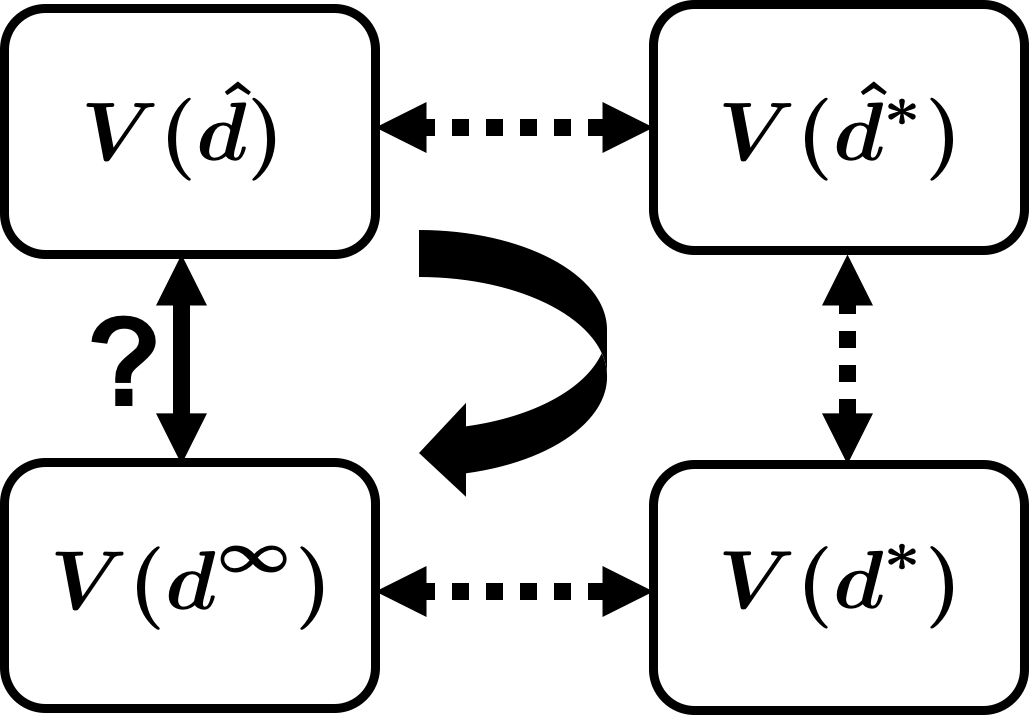}
		\caption{Strategy to calculate the gap between the equilibrium of the On-DS algorithm and that of Off-DS algorithm.} \label{fig:perf1}
	\end{center}
\end{figure}

We denote by $G^{\text{on}}$ the gap defined as the expected difference between the results of On-DS algorithm and Off-DS algorithm, i.e., $G^{\text{on}}  =  E[ V(\hat{d})-V(d^{\infty})]$, where $E$ denotes the expectation, and $\hat{d}$ and ${d}^{\infty}$ denote the average traffic profiles achieved by the On-DS algorithm and the offline-DS algorithm respectively.
It turns out that direct calculation of this gap is difficult. We therefore utilize two intermediate variables: ${d}^*$, the average traffic profile achieved by the R-ODS problem, and $\hat{d}^*$, the average traffic profile achieved by the relaxed online DS, i.e., the counterpart of R-ODS problem in the online scenario. Similar notations are applied to individual traffic profile $p_n$. With the relation shown in Fig.~\ref{fig:perf1}, we can write online gap as
\begin{eqnarray}
G^{\text{on}} \!\!& = &\!\! E\big[ V(\hat{d})\!-\!V(\hat{d}^*)\!+\!V(\hat{d}^*)\!-\!V({d}^*)\!+\!V({d}^*)\!-\!V({d}^{\infty}) \big] \nonumber\\
  & = &\!\! E\big[V(\hat{d})-V(\hat{d}^*)\big] + E\big[ V(\hat{d}^*)-V({d}^*)\big]\nonumber\\
  &&~+ E\big[V({d}^*)-V({d}^{\infty})\big].\label{eq:gap}
\end{eqnarray}

\begin{theorem}\label{the:peronline}
The gap, i.e., the expected difference between the results of On-DS algorithm and Off-DS algorithm is bounded as follows: 
\begin{align}
&G^{\text{on}}=E[ V(\hat{d})-V({d}^{\infty}) ]\leq\frac{2}{TN^2}\sum_{n\in\hN''}\!\|\hat{p}_n\|_2^2\nonumber\\[-5pt]
&~~~~~~~+\frac{(\delta\lambda)^2}{T}\sum_{t=2}^{T}\frac{1}{t}+\frac{\delta^2}{T^2}\sum_{t=0}^{T-1}F^2(t)\frac{T-t-1}{t+1}.\label{eq:peronline}
\end{align}
\hfill$\Box$
\end{theorem}
\begin{proof}
Applying the approach and results from Theorem~\ref{the:peroffline}, we have 
\begin{eqnarray}
0\leq E\big[V(\hat{d})-V(\hat{d}^*)\big]\leq \frac{2}{TN^2}\sum_{n\in\hN''}\!\|\hat{p}_n\|_2^2,\label{eq:firstterm}\\[-5pt]
-\frac{2}{TN^2}\sum_{n\in\hN''}\!\|\hat{p}_n\|_2^2\leq E\big[V({d}^*)-V({d}^{\infty})\big]\leq 0.
\end{eqnarray}
For the second term of (\ref{eq:gap}), under Assumptions~\ref{ass2}--\ref{ass3}, following \cite{gan2013real}, we get
\begin{eqnarray}
E\big[ V(\hat{d}^*)\!-\!V({d}^*)\big]=\frac{(\delta\lambda)^2}{T}\sum_{t=2}^{T}\frac{1}{t}+\frac{\delta^2}{T^2}\sum_{t=0}^{T-1}F^2(t)\frac{T-t-1}{t+1}.\!\!\!\!\!\!\!\!\!\!\!\!\!\!\!\!\!\!\nonumber\\[-5pt]\label{eq:midterm}
\end{eqnarray}
Combine (\ref{eq:firstterm})--(\ref{eq:midterm}) to obtain (\ref{eq:peronline}).
\end{proof}

Theorem~\ref{the:peronline} indicates that, the size of the gap between online and offline algorithms changes monotonically with prediction error of both base traffic and future arrival of deferrable traffic. Accordingly we can improve the result of On-DS algorithm by implementing better prediction mechanism, e.g., On-DS algorithm which updates its prediction to keep the value of prediction error small.
Also, if the impulse response $f$ is chosen to fade quickly enough, then as we have finer time granularity, we have $T\rightarrow\infty$, and $G^{\text{on}}\rightarrow 0$, which intuitively indicates that, with infinitely small timeslot, we can update our decisions frequently enough to mitigate prediction errors, and therefore have a negligible performance gap.
 
Lastly, similar to Theorem~\ref{the:peroffline}, define a relative gap $G_r^{\text{on}}:=G^{\text{on}}/V(d^{\infty})$. The following result is immediate. 
\begin{theorem}\label{the:onsub}
The relative gap $G_r^{\text{on}}$ diminishes as the number of discrete DAs $N''$ increases, i.e., 
\begin{eqnarray}
\lim_{N''\rightarrow\infty} G_r^{\text{on}}=0.\label{eq:ondimishsub}
\end{eqnarray}
 \hfill$\Box$
\end{theorem}

\begin{remark}
It is worth noting that equation (\ref{eq:ondimishsub}) does not necessarily imply a monotone decreasing of $G^{\text{on}}_r$ with respect to $N''$. This can be seen from Fig.~\ref{fig3_2} in Section~\ref{sec5} that does not show a decreasing $G^{\text{on}}_r$ as $N''$ increases.\hfill$\Box$
\end{remark}

By equations \eqref{eq:firstterm} and \eqref{eq:midterm}, it is straightforward to obtain  the following result. 
\begin{corollary}
The expected difference between the On-DS algorithm and the optimum of the R-ODS problem is bounded as follows:
\begin{align}
&E[ V(\hat{d})-V({d}^{*}) ]\leq\frac{2}{TN^2}\sum_{n\in\hN''}\!\|\hat{p}_n\|_2^2\nonumber\\[-5pt]
&~~~~~~~+\frac{(\delta\lambda)^2}{T}\sum_{t=2}^{T}\frac{1}{t}+\frac{\delta^2}{T^2}\sum_{t=0}^{T-1}F^2(t)\frac{T-t-1}{t+1}.\nonumber
\end{align} 
\hfill$\Box$
\end{corollary}

\section{Numerical Examples}\label{sec5}

In this section, we  provide numerical examples to evaluate the performance of the On-DS algorithm. We use certain composite traffic traces to drive  simulations to show the impact of base traffic prediction errors, deferrable traffic prediction errors, and deferrable traffic penetration levels. We expect the conclusions obtained to hold for real traffic.  

\subsection{Experimental Setup}\label{secVA}
Consider a 48-hour period of time starting from 4:00 pm to 3:59 pm two days later. We divide the 48 hours equally into 96 timeslots, each 30 minutes long. We consider scheduling traffic that arrives within the first 24 hours only, which may be allocated to the second 24 hours.

\subsubsection{Non-deferrable Traffic}\label{sec5A1}
The ``real" trace we use for non-deferrable traffic, or base traffic, is shown in Fig.~\ref{fig4_1} (red line). It is constructed by random fluctuation around the average base traffic trace (blue line) composed based on North American mobile web browsing activity by time of day in 2013 \cite{TRAFF}, shown in Fig.~\ref{fig1}. As modeled in Section~\ref{sec2A}, the prediction of base traffic follows (\ref{eq:ndp}), consisting of average base traffic $\bar{b}(\tau)$ and random deviation $\delta b_t(\tau)$ from the average value.  Following \cite{gan2013real}, at time $t$, $\delta b_t(\tau)$ is modeled as
\begin{eqnarray}
\delta b_t(\tau)=\sum_{s=t+1}^{\tau}\omega_s(\tau), ~t < \tau \leq T,
\end{eqnarray}
where $\omega_s(\tau)$ are random variables of Gaussian distribution with 0 mean and variances  
\begin{eqnarray}
E[\omega_s^2(\tau)]=\frac{\sigma^2}{\tau-s+1},~1\leq s\leq \tau\leq T.\label{basede}
\end{eqnarray}
In this way, $\delta b_t(\tau)$ has decreasing variance as $t$ approaches $\tau$, simulating a gradually improving prediction for some future timeslot $\tau$ as one gets closer to it. In simulation, we take the values of $\sigma^2$ in (\ref{basede}) from $0$ to $100$ with increment of 10, corresponding to a root-mean-square prediction error (RMSE) ranging from $0\%$ to $32\%$, looking 48 timeslots (24 hours) ahead.

\begin{figure}
	\centering
	\includegraphics[trim = 33mm 80mm 30mm 80mm, clip, scale=0.61]{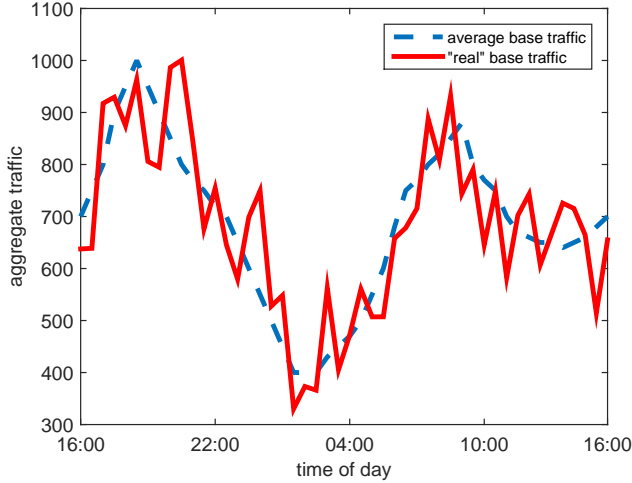}
	\caption{Base traffic: the average (blue/dotted) and a ``real" trace (red/solid).}
	\label{fig4_1}
\end{figure}

\begin{figure}
\centering
	\includegraphics[trim = 30mm 80mm 30mm 80mm, clip, scale=0.58]{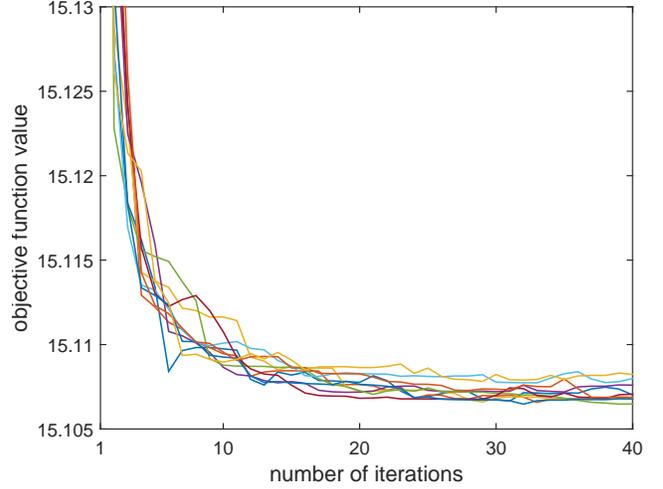}
	\caption{Repetitive experiments show that a number of 20 to 30 iterations give a satisfying result in terms of convergence.}
			\label{fig4_2}
\end{figure}	  

\subsubsection{Deferrable Traffic}\label{sec5A2}

We assume that the number of DAs arriving at each timeslot follows a ``shifted''  Poisson process $m+\text{poissrnd}(\lambda_p)$, with $m\geq0$ and $\text{poissrnd}(\lambda_p)$ denoting a Poisson process with rate $\lambda_p$. Here, we set $\lambda_p=4$, while each DA has a 50-50 chance to require continuous- or discrete-rate traffic. The total traffic $P_n$ of each DA is uniformly distributed in $[\underline{P},\overline{P}]$ where we set $\underline{P}=12$ and $\overline{P}=24$.
The deadline for DA $n$ is uniformly distributed in $[ t_n^a + l_n +\underline{D}, t_n^a + l_n +\overline{D} ]$, where $l_n = \ulcorner P_n/\overline{p}_n\urcorner$ is the minimum number of timeslots required by the DA calculated by ceiling function $ \ulcorner\cdot\urcorner$. We set $\underline{D}=6$, $\overline{D}=14$, and a universal bit rate upper bound $\overline{p}_n=3$. 

\subsubsection{Benchmarks for Comparison}\label{sec5A3}
We compare the performance of the On-DS algorithm with a few typical benchmarks to evaluate the impact of base traffic prediction error, the benefit of updating the prediction in real time, and the impact of deferrable traffic's penetration level. We thus consider the followings five cases in our experiments:
\begin{enumerate}
  \item[(0)]\label{enumerate1} {\em Offline demand shaping w/ Off-DS algorithm.} We use ``real" trace for future base traffic and use arrival information recorded from case (1) below for DAs. Applied with Off-DS algorithm, this case gives the optimal performance used as benchmark to characterize the gap of other cases.
  \item[(1)]\label{enumerate2} {\em Online demand shaping w/ On-DS algorithm.} We make prediction for both DAs' arrival and base traffic in the future. Prediction is updated at each timeslot. We run On-DS algorithm to schedule traffic.
  \item[(2)]\label{enumerate3} {\em Online demand shaping w/ exact information for base traffic and w/o exact information for DAs.} We use ``real" trace for base traffic and prediction for DAs. We apply On-DS algorithm. Comparison of case (2) with case (1) shows the impact of uncertainty in base traffic.
  \item[(3)]\label{enumerate4} {\em Demand shaping w/ updating prediction of base traffic and w/ exact information for DAs.} We use DAs arrival information recorded from case (1). Instead of applying virtual deferrable traffic, we schedule traffic profiles for all the future deferrable traffic. Since the exact base traffic information is not available, we updated base traffic prediction at each timeslot. Comparison of case (3) and (1) shows the impact of uncertainty in DAs arrival prediction.
  \item[(4)]\label{enumerate5} {\em Demand shaping w/o updating prediction of base traffic and w/ exact information for DAs.} We use prediction of the base traffic at the beginning ($t=1$) without further updating, and use arrival information recorded from case (1) for DAs. This case shows how the online algorithm benefits from updating prediction at each timeslot.
  
\end{enumerate}

We use the metric of relative gap $G_r(d)=(V(d)-V(d^{0}))/V(d^{0})$ to evaluate the performance, where $d^{0}$ is the results obtained from case (0). Also notice that when $d$ is calculated based on case (1), $G_r(d)$ becomes $G_r^{\text{on}}$ in Theorem~\ref{the:onsub}.

\subsection{Experiment Results}\label{sec5B}
Considering randomness in DAs' arrivals, base traffic prediction, and deciding traffic profiles for discrete DAs, we run simulation for 10 times, and take the average as the final result to present.

\subsubsection{Convergence Speed}We first run a case of randomly generated 143 continuous DAs and 150 discrete DAs by Off-DS algorithm with different numbers of iterations ranging from 1 to 40 for 10 times. Because of the random process in choosing traffic profiles for discrete DAs, we observe oscillation in objective function values for each individual run. However, the oscillation has a trend of diminishing as the more iterations are implemented, with satisfying enough results generated from running 20 to 30 iterations. See Fig.\ref{fig4_2} for the results. We will implement a number of 30 iterations to make each decision for the rest of simulation.

\subsubsection{Impact of Base Traffic Prediction Error}

As described in Section~\ref{secVA}, we can tune the variance $\sigma^2$ to emulate situations with different prediction errors in base traffic. As Fig.~\ref{fig3_1} shows, with updated prediction, case (1)'s performance is barely affected by the increasing prediction error,  keeping its relative gap under 5\%. This is almost as good as that of case (2) with perfect base traffic information. 
We can also see from the performance of case (3) the pure impact from prediction errors, while case (4) gives an example showing what happens if there is no updated prediction.

\subsubsection{Impact of Penetration Level of Discrete DAs}
 In this case, we fix the prediction error in base traffic at $\sigma^2=40$ and the average number of DAs' arrival at each timeslot at $\lambda_p=4$. We then tune the penetration level of discrete DAs from $25\%$ to $75\%$ with granularity of $5\%$. As shown in Fig.~\ref{fig3_2}, the relative gap maintains relatively unaffected by the changes of discrete DAs whose penetration has increased by three times. Here, we do \emph{not} observe a decreasing relative gap mainly because the gap is not monotonically decreasing with number of $N''$. 
 \begin{figure}
	\centering
		\includegraphics[trim = 0mm 0mm 0mm 0mm, clip, scale=0.58]{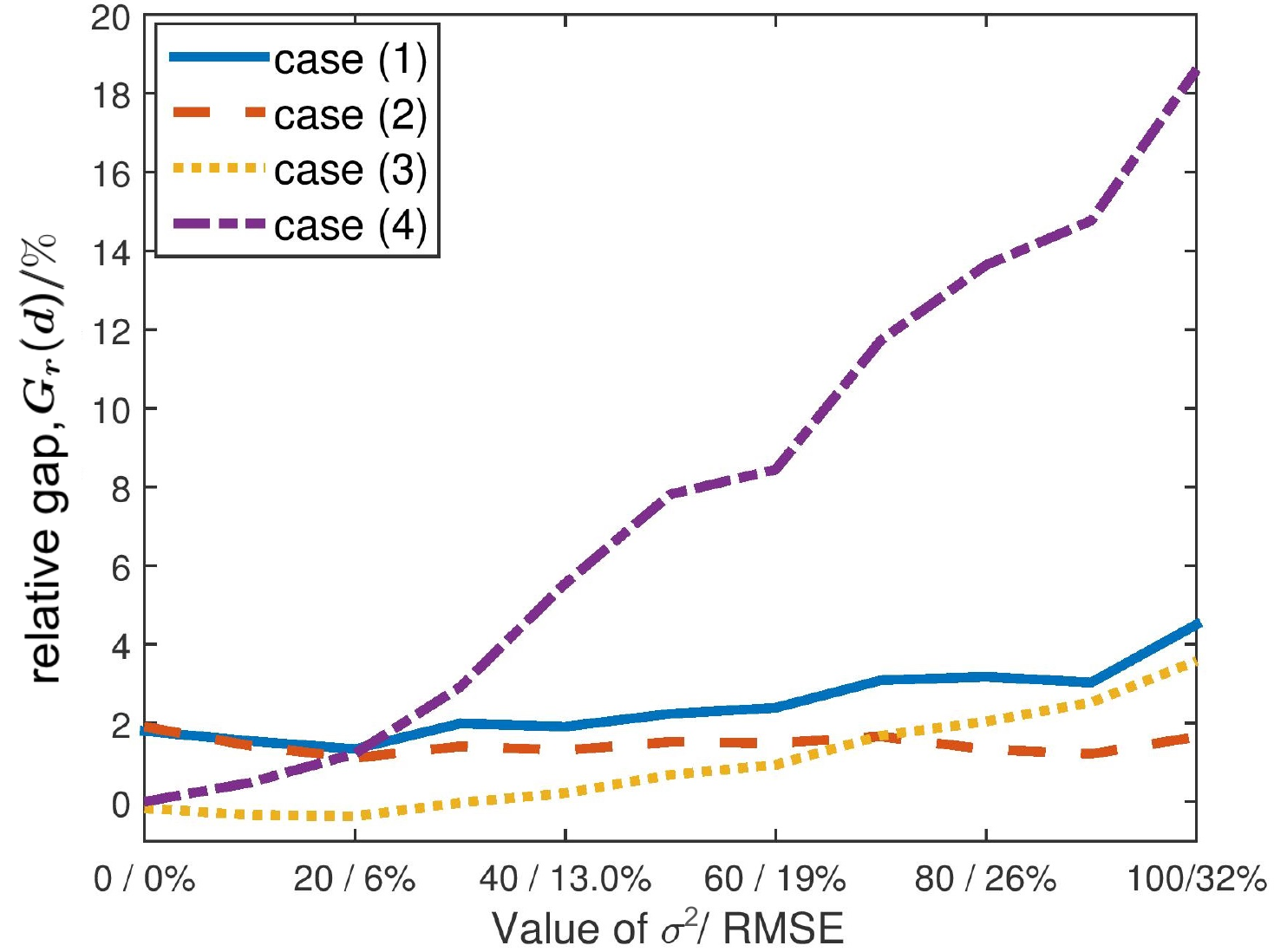}
		\caption{Base traffic prediction error has little impact on online algorithms with updated base traffic prediction.}
		\label{fig3_1}
\end{figure}
\vspace{10mm}
\begin{figure}
		\centering
		\includegraphics[trim = 0mm 0mm 0mm 0mm, clip, scale=0.58]{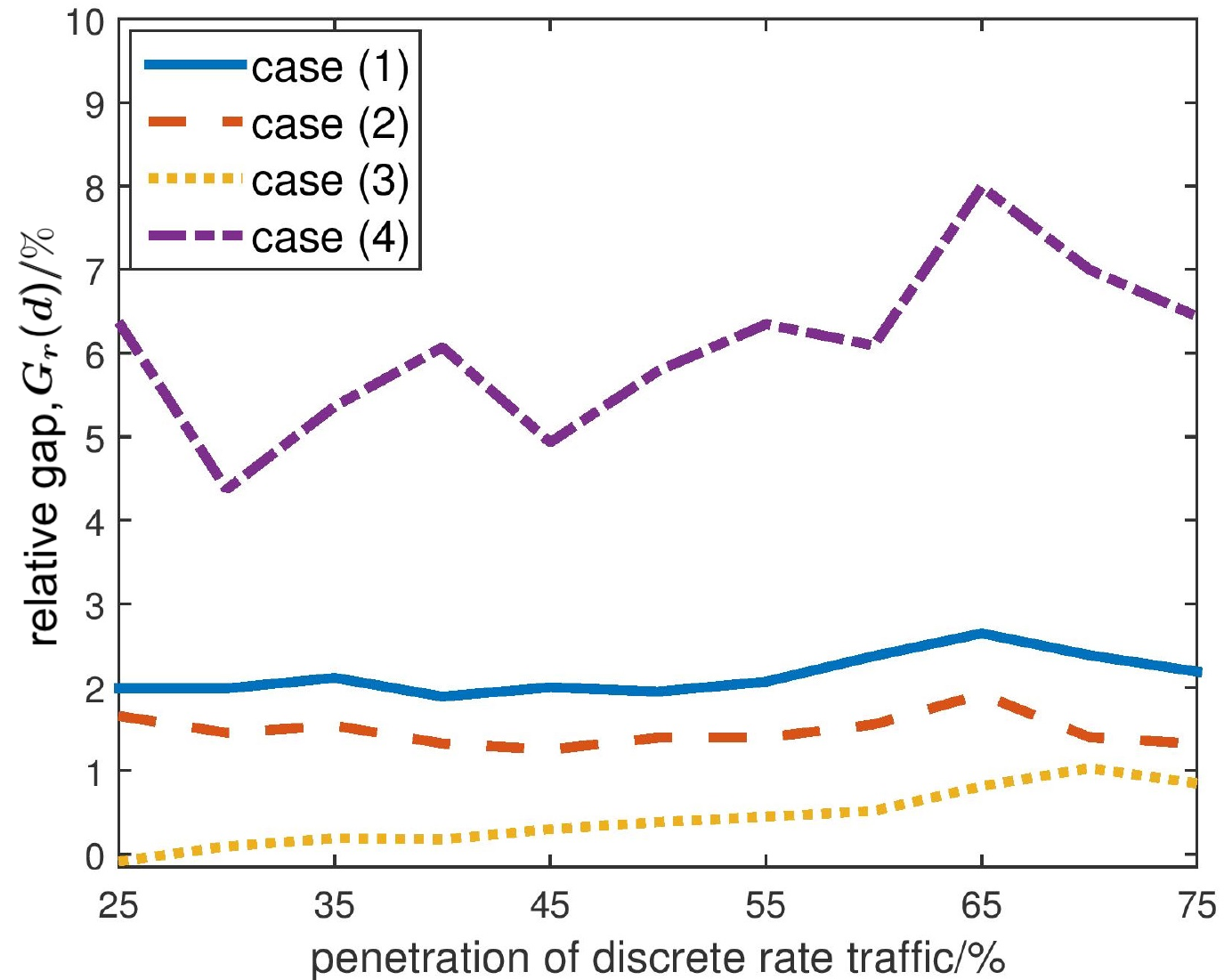}
		\caption{Increasing penetration level of deferrable traffic does not influence the relative gap of online algorithms. }
		\label{fig3_2}
\end{figure}	    

\section{Conclusion}\label{sec6}
We have formulated demand shaping in cellular networks as an optimization problem that minimizes the time variation in aggregate traffic subject to the rate and time requirements of the applications. We design a distributed and randomized offline demand shaping algorithm under complete traffic information and prove its almost surely convergence. We then consider a realistic setting with incomplete information where we can only predict future traffic to a certain degree of accuracy, and design an online demand shaping algorithm that updates the schedules of deferrable applications each time new information is available,  based on solving at each timeslot an optimization problem over a shrinking horizon from the current time to the end of the day. We compare the performance of the online algorithm against   the optimal offline algorithm analytically and numerically. As future work, we are investigating to integrate the incentive mechanisms such as the smart data pricing into the demand shaping algorithm design. We also plan to develop a platform to enable automatic demand shaping in cellular networks and investigate the related practical issues. 

\section*{Acknowledgement}

We would like to thank Seungil You for help with simulations and Lingwen Gan for careful comments.

%\IEEEtriggeratref{32} % This command breaks the reference manually so that no entry is broken.
\bibliographystyle{plain}
\bibliography{ref,ChenReferences,ChenReferencesPower,xinyang_ds}

\begin{thebibliography}{10}

\bibitem{airtelprice}
Airtel launches unlimited-usage night plans for calls, internet.
\newblock \url{http://businesstoday.intoday.in/story/
  airtel-night-plans-unlimited-usage -for-calls-internet/1/205272.html}.

\bibitem{att2014}
A{T}\&{T} still throttles ``unlimited data"---even when network not congested.
\newblock \url
  {https://arstechnica.com/information-technology/2014/12/att-still-throttles-unlimited-data-even-when-network-not-congested/}.

\bibitem{skyperate}
How much bandwidth does {S}kype need?
\newblock \url
  {https://support.skype.com/en/faq/FA1417/how-much-bandwidth-does-skype-need}.

\bibitem{netflixrate}
Netflix internet connection speed recommendations.
\newblock \url {https://help.netflix.com/en/node/306}.

\bibitem{sprintprice}
Sprint night and weekend minutes.
\newblock
  \url{http://shop2.sprint.com/en/stores/popups/voice_nights_weekends_7pm_popup.shtml}.

\bibitem{tmobile2016}
T-mobile now throttling mobile hotspots when network is congested.
\newblock \url
  {https://arstechnica.com/information-technology/2016/10/t-mobile-now-throttling-mobile-hotspots-when-network-is-congested/}.

\bibitem{telkomprice}
Telkom night surfer plan.
\newblock \url {http://www.telkommobile.co.za/plans/prepaid-data/60gbpromo/}.

\bibitem{verizonprice1}
Verizon nationwide for business plans.
\newblock
  \url{http://business.verizonwireless.com/content/b2b/en/shop-business-products/business-plans/nationwide-for-business.html}.

\bibitem{verizon2014}
Verizon wireless to slow down users with unlimited {4G LTE} plans.
\newblock \url
  {https://arstechnica.com/information-technology/2014/07/verizon-wireless-to-slow-down-users-with-unlimited-4g-lte-plans/}.

\bibitem{balasubramanian2010augmenting}
A.~Balasubramanian, R.~Mahajan, and A.~Venkataramani.
\newblock Augmenting mobile 3{G} using {W}i{F}i.
\newblock {\em Proceedings of International Conference on Mobile Systems,
  Applications, and Services}, pages 209--222, 2010.

\bibitem{Boyd}
S.~Boyd and L.~Vandenberghe.
\newblock {\em Convex Optimization}.
\newblock Cambridge University Press, 2004.

\bibitem{Chen-2012-DRR}
L.~Chen, L.~Jiang, N.~Li, and S.~H. Low.
\newblock Optimal demand response: Problem formulation and deterministic case.
\newblock {\em Control and Optimization Theory for Electric Smart Grids}, 2012.

\bibitem{cheung2017congestion}
M.~H. Cheung, F.~Hou, J.~Huang, and R.~Southwell.
\newblock Congestion-aware distributed network selection for integrated
  cellular and {Wi-Fi} networks.
\newblock {\em IEEE Journal on Selected Areas in Communications}, 35(6):1269
  --1281, 2017.

\bibitem{gan2012stochastic}
L.~Gan, U.~Topcu, and S.~H. Low.
\newblock Stochastic distributed protocol for electric vehicle charging with
  discrete charging rate.
\newblock {\em Proceedings of Power and Energy Society General Meeting}, pages
  1--8, 2012.

\bibitem{gan2013real}
L.~Gan, A.~Wierman, U.~Topcu, N.~Chen, and S.~H. Low.
\newblock Real-time deferrable load control: {H}andling the uncertainties of
  renewable generation.
\newblock {\em Proceedings of International Conference on Future Energy
  Systems}, pages 113--124, 2013.

\bibitem{Grimmett}
G.~R. Grimmett and D.~R. Stirzaker.
\newblock {\em Probability and Random Process}.
\newblock Oxford University Press, third edition, 2001.

\bibitem{tube2012survey}
S.~Ha, S.~Sen, C.~Joe-Wong, Y.~Im, and M.~Chiang.
\newblock Tube survey questions and demographics.
\newblock {\url{http://www.princeton.edu/~cjoe/TUBE_Survey.pdf}}, Jan 2012.

\bibitem{ha2012tube}
S.~Ha, S.~Sen, C.~Joe-Wong, Y~Im, and M.~Chiang.
\newblock Tube: {T}ime-dependent pricing for mobile data.
\newblock {\em ACM SIGCOMM Computer Communication Review}, 42(4):247--258,
  2012.

\bibitem{hajiesmaili2017incentivizing}
M.~H. Hajiesmaili, L.~Deng, M.~Chen, and Z.~Li.
\newblock Incentivizing device-to-device load balancing for cellular networks:
  An online auction design.
\newblock {\em IEEE Journal on Selected Areas in Communications},
  35(2):265--279, 2017.

\bibitem{index2017global}
Cisco Visual~Networking Index.
\newblock Global mobile data traffic forecast update, 2016--2021.
\newblock 2017.

\bibitem{TRAFF}
Chitika Insights.
\newblock Hour-by-hour examination: Smartphone, tablet, and desktop usage
  rates, 2013.

\bibitem{iosifidis2015double}
G.~Iosifidis, L.~Gao, J.~Huang, and L.~Tassiulas.
\newblock A double-auction mechanism for mobile data-offloading markets.
\newblock {\em IEEE/ACM Transactions on Networking (TON)}, 23(5):1634--1647,
  2015.

\bibitem{jiang2008time}
L.~Jiang, S.~Parekh, and J.~Walrand.
\newblock Time-dependent network pricing and bandwidth trading.
\newblock {\em Proceedings of IEEE Network Operations and Management Symposium Workshops}, pages 193--200, 2008.

\bibitem{joe2011time}
C.~Joe-Wong, S.~Ha, and M.~Chiang.
\newblock Time-dependent broadband pricing: Feasibility and benefits.
\newblock {\em Proceedings of International Conference on Distributed Computing Systems
  (ICDCS)}, pages 288--298, 2011.

\bibitem{RHC}
W.~H. Kwon and A.E. Pearson.
\newblock A modified quadratic cost problem and feedback stabilization of a
  linear system.
\newblock {\em IEEE Transactions on Automatic Control}, 22(5):838--842, Oct
  1977.

\bibitem{lee2014economics}
J.~Lee, Y.~Yi, S.~Chong, and Y.~Jin.
\newblock Economics of {W}i{F}i offloading: {T}rading delay for cellular
  capacity.
\newblock {\em IEEE Transactions on Wireless Communications}, 13(3):1540--1554,
  2014.

\bibitem{lee2010mobile}
K.~Lee, J.~Lee, Y.~Yi, I.~Rhee, and S.~Chong.
\newblock Mobile data offloading: {H}ow much can {W}i{F}i deliver?
\newblock {\em IEEE/ACM Transactions on Networking}, 21(2):536--550, 
2013.

\bibitem{li2015market}
N.~Li.
\newblock A market mechanism for electric distribution networks.
\newblock {\em Proceedings of IEEE Annual Conference on Decision and Control
  (CDC)}, pages 2276--2282, 2015.

\bibitem{LiCL11}
N.~Li, L.~Chen, and S.~H. Low.
\newblock Optimal demand response based on utility maximization in power
  networks.
\newblock {\em Proceedings of IEEE Power Engineering Society General Meeting},
  July 2011.

\bibitem{mehmeti2014worth}
F.~Mehmeti and T.~Spyropoulos.
\newblock Is it worth to be patient? {A}nalysis and optimization of delayed
  mobile data offloading.
\newblock {\em Proceedings of IEEE INFOCOM}, pages 2364--2372, 2014.

\bibitem{parise2014mean}
F.~Parise, M.~Colombino, S.~Grammatico, and J.~Lygeros.
\newblock Mean field constrained charging policy for large populations of
  plug-in electric vehicles.
\newblock {\em Proceedings of IEEE Annual Conference on Decision and Control
  (CDC)}, pages 5101--5106, 2014.

\bibitem{paschalidis2000congestion}
I.~C. Paschalidis and J.~N. Tsitsiklis.
\newblock Congestion-dependent pricing of network services.
\newblock {\em IEEE/ACM Transactions on Networking}, 8(2):171--184, 2000.

\bibitem{potra2000interior}
F.~A. Potra and S.~J. Wright.
\newblock Interior-point methods.
\newblock {\em Journal of Computational and Applied Mathematics},
  124(1):281--302, 2000.

\bibitem{sen2012pricing}
S.~Sen, C.~Joe-Wong, S.~Ha, and M.~Chiang.
\newblock Pricing data: A look at past proposals, current plans, and future
  trends.
\newblock {\em arXiv preprint arXiv:1201.4197}, 2012.

\bibitem{sen2013smart}
S.~Sen, C.~Joe-Wong, S.~Ha, and M.~Chiang.
\newblock Smart data pricing ({SDP}): {E}conomic solutions to network
  congestion.
\newblock {\em SIGCOMM eBook on Recent Advances in Networking}, 2013.

\bibitem{tadrous2013proactive}
J.~Tadrous, A.~Eryilmaz, and H.~El~Gamal.
\newblock Proactive resource allocation: {H}arnessing the diversity and
  multicast gains.
\newblock {\em IEEE Transactions on Information Theory}, 59(8):4833--4854,
  2013.

\bibitem{zhang2016smart}
L.~Zhang.
\newblock {\em Smart Data Pricing in Wireless Data Networks: {A}n Economic
  Solution to Congestion}.
\newblock PhD thesis, The Hong Kong Polytechnic University, 2016.

\bibitem{zhao2015peak}
S.~Zhao, X.~Lin, and M.~Chen.
\newblock Peak-minimizing online {EV} charging: Price-of-uncertainty and
  algorithm robustification.
\newblock {\em Proceedings of IEEE Conference on Computer Communications
  (INFOCOM)}, pages 2335--2343, 2015.

\bibitem{zhou2014demand}
X.~Zhou and L.~Chen.
\newblock Demand shaping in cellular networks.
\newblock {\em Proceedings of Annual Allerton Conference on Communication,
  Control, and Computing (Allerton)}, pages 621--628, 2014.

\bibitem{zhou2017online}
X.~Zhou, E.~Dall'Anese, L.~Chen, and A.~Simonetto.
\newblock An incentive-based online optimization framework for distribution
  grids.
\newblock {\em IEEE Transactions on Automatic Control}, 2017.

\end{thebibliography}

\end{document}